\documentclass[a4paper]{amsart}%
\usepackage{amscd}
\usepackage{amsthm}
\usepackage{graphicx}
\usepackage{amsmath}
\usepackage{amsfonts}
\usepackage{amssymb}%
\providecommand{\U}[1]{\protect\rule{.1in}{.1in}}
\newtheorem{theorem}{Theorem}

\newtheorem{definition}[theorem]{Definition}

\newtheorem{lemma}[theorem]{Lemma}

\newtheorem{proposition}[theorem]{Proposition}
\newtheorem{remark}[theorem]{Remark}

\newcommand\supp{\mathop{\rm supp}}
\newcommand\esssup{\mathop{\rm ess \, sup}}
\newcommand\essinf{\mathop{\rm ess \, inf}}

\begin{document}

\title[Estimates for Riesz potential on weighted variable Hardy spaces]{Estimates for Riesz potential on weighted variable Hardy spaces revisited}
\author{Pablo Rocha}
\address{Departamento de Matem\'atica, Universidad Nacional del Sur, Bah\'{\i}a Blanca, 8000 Buenos Aires, Argentina.}
\email{pablo.rocha@uns.edu.ar}
\thanks{\textbf{2.020 Math. Subject Classification}: 42B30, 42B25, 42B35, 46E30}

\begin{abstract}
In [Math. Ineq. \& appl., Vol 26 (2) (2023), 511-530] and [Period. Math. Hung., 89 (1) (2024), 116-128], the present author proved that the Riesz potential $I_{\alpha}$ extends to a bounded operator $H^{p(\cdot)}_{\omega}(\mathbb{R}^n) \to L^{q(\cdot)}_{\omega}(\mathbb{R}^n)$ and 
$H^{p(\cdot)}_{\omega}(\mathbb{R}^n) \to H^{q(\cdot)}_{\omega}(\mathbb{R}^n)$ respectively, under the following two assumptions:

$A1)$ $\omega \in \mathcal{W}_{q(\cdot)}$ with $q(\cdot) \in \mathcal{P}^{\log}(\mathbb{R}^{n})$ and $\frac{1}{p(\cdot)} := \frac{1}{q(\cdot)} + \frac{\alpha}{n}$;

$A2)$ for every cube $Q \subset \mathbb{R}^{n}$, 
$\| \chi_Q \|_{L^{q(\cdot)}_{\omega}} \approx |Q|^{-\alpha/n} \| \chi_Q \|_{L^{p(\cdot)}_{\omega}}$. \\
In this note, we re-establish such estimates for $I_{\alpha}$ without assuming the hypothesis $A2)$. These proofs are simpler than the previous ones. 
\end{abstract}

\maketitle

\section{Introduction}

Kwok-Pun Ho in \cite{Ho} developed the weighted theory for variable Hardy spaces on $\mathbb{R}^n$, which are denoted by $H^{p(\cdot)}_{\omega}(\mathbb{R}^{n})$. He established the atomic decompositions for $H^{p(\cdot)}_{\omega}(\mathbb{R}^{n})$ and also gave a maximal characterization for these spaces. Moreover, he revealed some intrinsic structures of atomic decomposition for Hardy type spaces. His results generalize the infinite atomic decompositions obtained in \cite{Bui, Garcia, Nakai, Stein, Torch}.

The \textit{Riesz potential} $I_{\alpha }$ of order $\alpha \in (0, n)$ is defined, say on $\mathcal{S}(\mathbb{R}^{n})$, by
\begin{equation} \label{Ia}
I_{\alpha }f(x)=\int_{\mathbb{R}^{n}} f(y) |x-y|^{\alpha - n}dy, \,\,\, x \in \mathbb{R}^{n}.
\end{equation}
With respect to the behavior of the operator $I_{\alpha}$ on Hardy type spaces, E. Stein and G. Weiss \cite{steinweiss} proved the 
$H^{p}(\mathbb{R}^n) \to H^{q}(\mathbb{R}^n)$ boundedness of $I_{\alpha}$ for $\frac{n-1}{n} < p \leq 1$ and 
$\frac{1}{q} = \frac{1}{q} - \frac{\alpha}{n}$.  The lower bound $\frac{n-1}{n}$ is because these authors described the $H^p$ theory in terms of systems of conjugate harmonic functions (see \cite[\S 5.16]{Stein}). Afterwards, M. Taibleson and G. Weiss \cite{T-W} obtained, using a molecular decomposition for elements in $H^p$, the boundedness of the Riesz potential $I_{\alpha }$ from 
$H^{p}(\mathbb{R}^{n})$ into $H^{q}(\mathbb{R}^{n})$ for $0 < p \leq 1$ and $\frac{1}{q}=\frac{1}{p}-\frac{\alpha }{n}$; independently S. Krantz obtained the same result in \cite{krantz}. The $H^p(w^p) \to H^q(w^q)$ boundedness for $I_{\alpha}$ was proved by J. Str\"omberg and R. Wheeden \cite{wheeden} (see also \cite{Gatto, Rocha3}). P. Rocha and M. Urciuolo \cite{Rocha-Ur} established the $H^{p(\cdot)} \to H^{q(\cdot)}$ boundedness of $I_{\alpha}$, where the exponents $p(\cdot)$ and $q(\cdot)$ are related by 
$\frac{1}{p(\cdot)} - \frac{1}{q(\cdot)} = \frac{\alpha}{n}$ and $p(\cdot) \in \mathcal{P}^{\log}(\mathbb{R}^{n})$ (see also \cite{Rocha2}). 

In \cite{Rocha4} and \cite{Rocha5}, the present author proved that the operator $I_{\alpha}$ extends to a bounded operator 
$H^{p(\cdot)}_{\omega}(\mathbb{R}^{n}) \to L^{q(\cdot)}_{\omega}(\mathbb{R}^{n})$ and $H^{p(\cdot)}_{\omega}(\mathbb{R}^{n}) \to H^{q(\cdot)}_{\omega}(\mathbb{R}^{n})$ respectively, for $\frac{1}{p(\cdot)} := \frac{1}{q(\cdot)} + \frac{\alpha}{n}$; under the following two assumptions:

$A1)$ $\omega \in \mathcal{W}_{q(\cdot)}$ with $q(\cdot) \in \mathcal{P}^{\log}(\mathbb{R}^{n})$ (see Definitions \ref{pesos Wp} and \ref{Plog} below); and

$A2)$ for every cube $Q \subset \mathbb{R}^{n}$, 
$\| \chi_Q \|_{L^{q(\cdot)}_{\omega}} \approx |Q|^{-\alpha/n} \| \chi_Q \|_{L^{p(\cdot)}_{\omega}}$. \\
More precisely, in \cite{Rocha4} we proved an "off-diagonal" version of the Fefferman-Stein vector-valued maximal inequality on weighted variable Lebesgue spaces. Then, by means of the atomic decomposition established in \cite{Ho} and \cite[Lemma 5.4]{Ho}, together with $A1)$ and $A2)$ we obtained the $H^{p(\cdot)}_{\omega}(\mathbb{R}^{n}) \to L^{q(\cdot)}_{\omega}(\mathbb{R}^{n})$ boundedness of $I_{\alpha}$. In \cite{Rocha5}, we proved a molecular reconstruction theorem for $H^{p(\cdot)}_{\omega}(\mathbb{R}^n)$. This result, the atomic decomposition for $H^{p(\cdot)}_{\omega}(\mathbb{R}^{n})$, $A1)$ and $A2)$ allowed us to obtain the $H^{p(\cdot)}_{\omega}(\mathbb{R}^{n}) \to H^{q(\cdot)}_{\omega}(\mathbb{R}^{n})$ boundedness of $I_{\alpha}$.

We point out that if $q(\cdot) \in \mathcal{P}^{\log}(\mathbb{R}^{n})$ and $\omega \equiv 1$, then the condition $A2)$ holds. This was observed in \cite{Rocha-Ur}. In \cite{Rocha4}, the author gave non trivial examples of power weights satisfying $A2)$. So, the condition $A2)$ is an admissible hypothesis (see Remark \ref{hyp A2} below).

The purpose of this article is to re-establish such estimates for the Riesz potential $I_{\alpha}$ without assuming the hypothesis $A2)$. To avoid the assumption $A2)$ (and the use of \cite[Lemma 5.4]{Ho}), we will follow and adapt some ideas of the article \cite{Uribe} to our context. More precisely, in Section 3, we establish the following two vector-valued inequalities in the weighted variable setting:
\begin{equation} \label{ineq 1}
\left\Vert  \sum_{j=1}^{\infty}  h_j \right\Vert_{L^{q(\cdot)}_{\omega}(\mathbb{R}^n)} \lesssim 
\left\Vert \sum_{j=1}^{\infty} \left( \frac{1}{|Q_j|}\int_{Q_j} h_j^{q_0} \right)^{1/q_0} \chi_{Q_j}  \right\Vert_{L^{q(\cdot)}_{\omega}(\mathbb{R}^n)},
\end{equation}
where the $Q_j$'s are cubes, $\supp(h_j) \subset Q_j$, $q(\cdot)$ is an exponent, $q_0 > \max \{ 1, q_{+} \}$, and $\omega$ belongs to the weights class $\mathcal{W}_{q(\cdot)}$; and
\begin{equation} \label{ineq 2}
\left\Vert  \sum_{j=1}^{\infty} \lambda_j |Q_j|^{\frac{\alpha}{n}} \chi_{Q_j} \right\Vert_{L^{q(\cdot)}_{\omega}(\mathbb{R}^n)} \lesssim
\left\Vert \sum_{j=1}^{\infty} \lambda_j \chi_{Q_j}  \right\Vert_{L^{p(\cdot)}_{\omega}(\mathbb{R}^n)},
\end{equation}
where $0 < \alpha < n$, $\lambda_j > 0$, $q(\cdot)$ is an exponent, $\frac{1}{p(\cdot)} := \frac{1}{q(\cdot)} + \frac{\alpha}{n}$, and 
$\omega \in \mathcal{W}_{q(\cdot)}$. Inequalities of the type (\ref{ineq 1}) and (\ref{ineq 2}), for $q(\cdot) \equiv$ costant and certain weights $\omega$, were first considered by Grafakos and Kalton \cite{Grafakos}, and Str\"omberg and Wheeden \cite{wheeden}, respectively. The inequalities (\ref{ineq 1}) and (\ref{ineq 2}) generalize to the ones given in \cite[Lemma 4.9]{Uribe} and \cite[Lemma 4.11]{Uribe}. Then, with (\ref{ineq 1}) and (\ref{ineq 2}) together with some results of Sections 2 and 3, we will prove the following theorem in Section 4.

\

{\bf Theorem \ref{Hpw-Lqw and Hpw-Hqw}.} \textit{Let $0 < \alpha < n$, $q(\cdot) \in \mathcal{P}^{\log}(\mathbb{R}^{n})$ with 
$0 < q_{-} \leq q_{+} < \infty$, and $\omega \in \mathcal{W}_{q(\cdot)}$. If $\frac{1}{p(\cdot)} := \frac{1}{q(\cdot)} + \frac{\alpha}{n}$, then the Riesz potential $I_{\alpha}$ given by (\ref{Ia}) can be extended to a bounded operator $H^{p(\cdot)}_{\omega}(\mathbb{R}^{n}) \to L^{q(\cdot)}_{\omega}(\mathbb{R}^{n})$ and $H^{p(\cdot)}_{\omega}(\mathbb{R}^{n}) \to H^{q(\cdot)}_{\omega}(\mathbb{R}^{n})$.}

\

Instead of the molecular decomposition given in \cite{Rocha5}, here we will use the maximal characterization of 
$H^{p(\cdot)}_{\omega}(\mathbb{R}^{n})$ established in \cite{Ho} to obtain the 
$H^{p(\cdot)}_{\omega}(\mathbb{R}^{n}) \to H^{q(\cdot)}_{\omega}(\mathbb{R}^{n})$ boundedness of $I_{\alpha}$.

\

\textbf{Notation:} The symbol $A \lesssim B$ stands for the inequality $A \leq cB$ for some constant $c$. The symbol $A \approx B$ 
stands for $B \lesssim A \lesssim B$. We denote by $Q( x, r)$ the cube centered at $x \in \mathbb{R}^{n}$ with side lenght 
$r$. Given $\gamma >0$ and a cube $Q = Q(x, r)$, we set $\gamma Q = Q(x, \gamma r)$. For a measurable subset 
$E \subset \mathbb{R}^{n}$ we denote $|E|$ and $\chi_E$ the Lebesgue measure of $E$ and the characteristic function 
of $E$ respectively. Given a real number $s \geq 0$, we write $\lfloor s \rfloor$ for the integer part of $s$. As usual we denote with 
$\mathcal{S}(\mathbb{R}^{n})$ the space of smooth and rapidly decreasing functions, with $\mathcal{S}'(\mathbb{R}^{n})$  the dual space. 
If $\beta$ is the multiindex $\beta=(\beta_1, ..., \beta_n)$, then $|\beta| = \beta_1 + ... + \beta_n$. 

Throughout this paper, $C$ will denote a positive constant, not necessarily the same at each occurrence.

\section{Preliminaries} 

For $0 \leq \alpha < n$, we define the \textit{fractional maximal operator} $M_{\alpha}$ by
\[
M_{\alpha}f(x) = \sup_{Q \ni x} |Q|^{\frac{\alpha}{n} - 1}\int_{Q} |f(y)| \, dy,
\]
where $f$ is a locally integrable function on $\mathbb{R}^{n}$ and the supremum is taken over all the cubes $Q$ containing $x$. For 
$\alpha=0$, we have that $M_0 = M$, where $M$ is the {\it Hardy-Littlewood maximal operator} on $\mathbb{R}^{n}$.

Let $p(\cdot) : \mathbb{R}^{n} \to (0, \infty)$ be a measurable function. Given a measurable set $E$, let
\[
p_{-}(E) = \essinf_{ x \in E } p(x), \,\,\,\, \text{and} \,\,\,\, p_{+}(E) = \esssup_{x \in E} p(x).
\]
When $E = \mathbb{R}^{n}$, we will simply write $p_{-} := p_{-}(\mathbb{R}^{n})$ and $p_{+} := p_{+}(\mathbb{R}^{n})$. We also define 
$\underline{p} := \min \{ 1, p_{-} \}$.

Given a measurable function $f$ on $\mathbb{R}^{n}$, define the modular $\rho$ associated with $p(\cdot)$ by
\[
\rho(f) = \int_{\mathbb{R}^{n}} |f(x)|^{p(x)} dx.
\]
We define the variable Lebesgue space $L^{p(\cdot)} = L^{p(\cdot)}(\mathbb{R}^{n})$ to be the set of all measurable functions $f$ such that, for some $\lambda > 0$, $\rho \left( f/\lambda \right) < \infty$. This becomes a quasi normed space when equipped with the Luxemburg norm
\[
\| f \|_{L^{p(\cdot)}} = \inf \left\{ \lambda > 0 : \rho \left( f/\lambda \right) \leq 1 \right\}.
\]

Given a weight $\omega$, i.e.: a locally integrable function on $\mathbb{R}^{n}$ such that $0 < \omega(x) < \infty$ almost everywhere, we define the weighted variable Lebesgue space $L^{p(\cdot)}_{\omega}$ as the set of all measurable functions 
$f : \mathbb{R}^{n} \to \mathbb{C}$ such that $\| f \omega \|_{L^{p(\cdot)}} < \infty$. If $f \in L^{p(\cdot)}_{\omega}$, we define 
its quasi-norm by
\begin{equation} \label{quasi}
\| f \|_{L^{p(\cdot)}_\omega} := \| f \omega \|_{L^{p(\cdot)}}. 
\end{equation}

The following result follows from the definition of the $L^{p(\cdot)}_{\omega}$-norm.

\begin{lemma} \label{potencia s}
Given a measurable function $p(\cdot) : \mathbb{H}^{n} \to (0, \infty)$ with $0 < p_{-} \leq p_{+} < \infty$ and a weight $\omega$, then 
\\
(i) $\| f \|_{L^{p(\cdot)}_{\omega}} \geq 0$ and $\| f \|_{L^{p(\cdot)}_{\omega}} = 0$ if and only if $f \equiv 0$ a.e.,
\\
(ii) $\| c f \|_{L^{p(\cdot)}_{\omega}} = |c| \| f \|_{L^{p(\cdot)}_{\omega}}$ for all $f \in L^{p(\cdot)}_{\omega}$ and all $c \in \mathbb{C}$,
\\
(iii) $\| f +  g \|_{L^{p(\cdot)}_{\omega}} \leq 2^{1/\underline{p} - 1}(\| f \|_{L^{p(\cdot)}_{\omega}} + \| g \|_{L^{p(\cdot)}_{\omega}})$ for all $f, g \in L^{p(\cdot)}_{\omega}$,
\\
(iv) $\| f \|_{L^{p(\cdot)}_{\omega}}^{s} = \| |f|^{s} \|_{L^{p(\cdot)/s}_{\omega^s}}$ for every $s > 0$.
\end{lemma}

For a measurable function $p(\cdot) : \mathbb{R}^{n} \to [1, \infty)$, its conjugate function $p'(\cdot)$ is defined by 
$\frac{1}{p(x)} + \frac{1}{p'(x)} = 1$. We have the following equivalent expression for the $L^{p(\cdot)}_{\omega}$-norm.

\begin{proposition} \label{norma equivalente}
Let $p(\cdot) : \mathbb{R}^{n} \to [1, \infty)$ be a measurable function and $\omega$ be a locally integrable function such that 
$0 < \omega(x) < \infty$ almost everywhere. Then
\[
\| f \|_{L^{p(\cdot)}_{\omega}} \approx \sup \left\{ \int_{\mathbb{R}^{n}} |f(x) g(x)| dx : \| g \|_{L^{p'(\cdot)}_{\omega^{-1}}} \leq 1
\right\}.
\]
\end{proposition}

\begin{proof} The proposition follows from \cite[Corollary 3.2.14]{Diening}.
\end{proof}

Next, we introduce the weights used in \cite{Ho} to define weighted Hardy spaces with variable exponents.

\begin{definition} (See \cite[Remark 1]{Rocha4}) \label{pesos Wp}
Let $p(\cdot) : \mathbb{R}^{n} \to (0, \infty)$ be a measurable function with $0 < p_{-} \leq p_{+} < \infty$. We define
$\mathcal{W}_{p(\cdot)}$ as the set of all weights $\omega$ such that
\[
(i) \,\, \text{there exists} \,\, 0 < p_{\ast} < \min \{1, p_{-} \} \,\, \text{such that} \,\, 
\| \chi_Q \|_{L^{p(\cdot)/p_{\ast}}_{\omega^{p_{\ast}}}} < \infty, \,\,\, \text{and}
\]
\[ 
\| \chi_Q \|_{L^{(p(\cdot)/p_{\ast})'}_{\omega^{-p_{\ast}}}} < \infty, \,\,\, \text{for all cube} \,\, Q;
\]
\[
(ii) \,\, \text{there exist $\kappa > 1$ and $s > \max \{ 1, 1/p_{-} \}$ such that Hardy-Littlewood maximal}
\]
\[ 
\text{operator $M$ is bounded on $L^{(sp(\cdot))'/\kappa}_{\omega^{-\kappa/s}}$}. 
\]
\end{definition}

Before stating the definition of weighted variable Hardy spaces $H^{p(\cdot)}_{\omega}(\mathbb{R}^{n})$, we introduce two indices, which 
are related to the intrinsic structure of the atomic decomposition of $H^{p(\cdot)}_{\omega}(\mathbb{R}^{n})$ established in \cite{Ho}. Given 
$\omega \in \mathcal{W}_{p(\cdot)}$, we write
\[
s_{\omega, \, p(\cdot)} := \inf \left\{ s > \max \{ 1, 1/p_{-} \} : M \,\, \text{is bounded on} \,\, L^{(sp(\cdot))'}_{\omega^{-1/s}} \right\}
\]
and
\[
\mathbb{S}_{\omega, \, p(\cdot)} := \left\{ s > \max \{ 1, 1/p_{-} \} : M \,\, \text{is bounded on} \,\, L^{(sp(\cdot))'/\kappa}_{\omega^{-\kappa/s}} \,\,
\text{for some} \,\, \kappa > 1 \right\}.
\]
Then, for every fixed $s \in \mathbb{S}_{\omega, \, p(\cdot)}$, we define
\[
\kappa_{\omega, \, p(\cdot)}^{s} := \sup \left\{ \kappa > 1 : M \,\, \text{is bounded on} \,\, L^{(sp(\cdot))'/\kappa}_{\omega^{-\kappa/s}} \right\}.
\]
The index $\kappa_{\omega, \, p(\cdot)}^{s}$ is used to measure the left-openness of the boundedness of $M$ on the family 
$\left\{ L^{(sp(\cdot))'/\kappa}_{\omega^{-\kappa/s}} \right\}_{\kappa > 1}$. The index $s_{\omega, \, p(\cdot)}$ is related to the vanishing moment condition and the index $\kappa_{\omega, \, p(\cdot)}^{s}$ is related to the size condition of the atoms 
(see \cite[Theorems 5.3 and 6.3]{Ho}). 

\begin{proposition} (\cite[Proposition 3]{Rocha4})\label{WqcWp}
Let $0 < \alpha < n$ and let $q(\cdot) : \mathbb{R}^{n} \to (0, \infty)$ be a measurable function such that $0 < q_{-} \leq q_{+} < \infty$. If $\omega \in \mathcal{W}_{q(\cdot)}$ and $\frac{1}{p(\cdot)} := \frac{1}{q(\cdot)} + \frac{\alpha}{n}$, then 
$\omega \in \mathcal{W}_{p(\cdot)}$. Moreover, $s_{\omega, \, p(\cdot)} \leq s_{\omega, \, q(\cdot)} + \frac{\alpha}{n}$.
\end{proposition}

For a measurable function $p(\cdot) : \mathbb{R}^{n} \to (0, \infty)$ such that $0 < p_{-}\leq p_{+} < \infty$ and 
$\omega \in \mathcal{W}_{p(\cdot)}$, in \cite{Ho} the author give a variety of distinct approaches, based on
differing definitions, all lead to the same notion of weighted variable Hardy space $H^{p(\cdot)}_{\omega}$. 

We recall some terminologies and notations from the study of maximal functions. Given $N \in \mathbb{N}$ and 
$\phi \in \mathcal{S}(\mathbb{R}^n)$, define
\[
\| \phi \|_{\mathcal{S}(\mathbb{H}^{n}), \, N}:= \sum\limits_{\left\vert \mathbf{\beta }\right\vert \leq N}\sup\limits_{x\in \mathbb{R}^{n}}\left( 1+\left\vert x\right\vert \right) ^{N}\left\vert \partial^{
\mathbf{\beta }}\phi (x)\right\vert
\]
and
\[
\mathcal{F}_{N}=\left\{ \varphi \in \mathcal{S}(\mathbb{R}^{n}): \| \varphi \|_{\mathcal{S}(\mathbb{H}^{n}), \, N} \leq 1\right\}.
\]
For any $f \in \mathcal{S}'(\mathbb{R}^{n})$, the grand maximal function of $f$ is given by 
\[
\mathcal{M}_N f(x)=\sup\limits_{t>0}\sup\limits_{\varphi \in \mathcal{F}_{N}}\left\vert \left( \varphi_t \ast f\right)(x) \right\vert,
\]
where $\varphi_t(x) = t^{-n} \varphi(t^{-1} x)$.

\begin{definition} \label{Def Hpw} (See \cite[Theorem 6.1]{Ho}) 
Let $p(\cdot):\mathbb{R}^{n} \to ( 0,\infty)$, $0 < p_{-} \leq p_{+} < \infty$, $\omega \in \mathcal{W}_{p(\cdot)}$ and 
$N \geq \lfloor n s_{\omega, \, p(\cdot)} - n \rfloor$ fixed. The weighted variable Hardy space $H^{p(.)}_{\omega}(\mathbb{R}^{n})$ is the 
set of all $f \in \mathcal{S}'(\mathbb{R}^{n})$ for which $\| \mathcal{M}_N f \|_{L^{p(\cdot)}_{\omega}} < \infty$. In this case we define 
$\| f \|_{H^{p(\cdot)}_{\omega}} := \| \mathcal{M}_N f \|_{L^{p(\cdot)}_{\omega}}$.
\end{definition}

\begin{definition} \label{def atom} Let $p(\cdot):\mathbb{R}^{n} \to ( 0,\infty)$, $0 < p_{-} \leq p_{+} < \infty $, $p_{0} > 1$, and 
$\omega \in \mathcal{W}_{p(\cdot)}$. Fix an integer $N \geq 1$. A function $a(\cdot)$ 
on $\mathbb{R}^{n}$ is called a $\omega-(p(\cdot), p_{0}, N)$ atom if there exists a cube $Q$ such that

$a_{1})$ $\supp ( a ) \subset Q$,

$a_{2})$ $\| a \|_{L^{p_{0}}}\leq \frac{| Q |^{\frac{1}{p_{0}}}}{\| \chi _{Q} \|_{L^{p(\cdot)}_{\omega}}}$,

$a_{3})$ $\displaystyle{\int} x^{\beta}  a(x) \, dx = 0$ for all $| \beta | \leq N$.
\end{definition}

The following theorem is a version of the atomic decomposition for $H^{p(\cdot)}_{\omega}$ obtained in \cite{Ho}.

\begin{theorem} \label{w atomic decomp}
Let $1 < p_0 < \infty$, $p(\cdot) : \mathbb{R}^{n} \to (0, \infty)$ be a measurable function with $0 < p_{-} \leq p_{+} < \infty$ and 
$\omega \in \mathcal{W}_{p(\cdot)}$. Then, for every $f \in H^{p(\cdot)}_{\omega}(\mathbb{R}^{n}) \cap L^{p_0}(\mathbb{R}^{n})$ and
every integer $N \geq \lfloor n s_{\omega, \, p(\cdot)} - n \rfloor$ fixed, there exist a sequence of scalars 
$\{ \lambda_j \}_{j=1}^{\infty}$, a sequence of cubes $\{ Q_j \}_{j=1}^{\infty}$ and $\omega - (p(\cdot), p_0, N)$ atoms $a_j$ supported 
on $Q_j$ such that $f = \displaystyle{\sum_{j=1}^{\infty} \lambda_j a_j}$ converges in $L^{p_0}(\mathbb{R}^{n})$ and
\begin{equation} \label{Hpw norm atomic}
\left\| \sum_{j=1}^{\infty} \left( \frac{|\lambda_j |}{\| \chi_{Q_j} \|_{L^{p(\cdot)}_{\omega}}} \right)^{\theta} \chi_{Q_j} \right\|_{L^{p(\cdot)/\theta}_{\omega^{\theta}}}^{1/\theta}
\lesssim \| f \|_{H^{p(\cdot)}_{\omega}}, \,\,\, \text{for all} \,\,\, 0 < \theta < \infty,
\end{equation}
where the implicit constant in (\ref{Hpw norm atomic}) is independent of $\{ \lambda_j \}_{j=1}^{\infty}$, $\{ Q_j \}_{j=1}^{\infty}$, and 
$f$.
\end{theorem}

\begin{proof}
The existence of a such atomic decomposition is guaranteed by \cite[Theorem 6.2]{Ho}. Its construction is analogous to that given for classical Hardy spaces (see \cite[Chapter III]{Stein}). So, following the proof in \cite[Theorem 3.1]{Rocha1}, we obtain the convergence of the atomic series to $f$ in $L^{p_0}(\mathbb{R}^{n})$.
\end{proof}

\begin{definition} \label{Plog}
We say that an exponent function $p(\cdot) : \mathbb{R}^{n} \to (0, \infty)$ such that $0 < p_{-} \leq p_{+} < \infty$ belongs 
to $\mathcal{P}^{\log}(\mathbb{R}^{n})$, if there exist two positive constants $C$ and $C_{\infty}$ such that $p(\cdot)$ satisfies the 
local log-H\"older continuity condition, i.e.:
\[
|p(x) - p(y)| \leq \frac{C}{-\log(|x-y|)}, \,\,\, |x-y| \leq \frac{1}{2},
\]
and is log-H\"older continuous at infinity, i.e.:
\[
|p(x) - p_{\infty}| \leq \frac{C_{\infty}}{\log(e+|x|)}, \,\,\, x \in \mathbb{R}^{n},
\]
for some $p_{-} \leq p_{\infty} \leq p_{+}$.
\end{definition}

We define the set $\mathcal{S}_{0}(\mathbb{R}^{n})$ by
\[
\mathcal{S}_{0}(\mathbb{R}^{n}) =\left\{ \varphi \in \mathcal{S}(\mathbb{R}^{n}) : \int x^{\beta} \varphi(x) dx = 0, \,\, \text{for all} \, 
\beta \in \mathbb{N}_{0}^{n} \right\}.
\] 

\begin{proposition} (\cite[Proposition 2.1]{VHo}) \label{dense}
Let $p(\cdot) \in \mathcal{P}^{\log}(\mathbb{R}^{n})$ with $0 < p_{-} \leq p_{+} < \infty$. If $\omega \in \mathcal{W}_{p(\cdot)}$, then
$\mathcal{S}_{0}(\mathbb{R}^{n}) \subset H^{p(\cdot)}_{\omega}(\mathbb{R}^{n})$ densely.
\end{proposition}

\section{Auxiliary results}

The following four lemmas will allow us to obtain our main result. These lemmas do not require the assumption 
$\| \chi_Q \|_{L^{q(\cdot)}_{\omega}} \approx |Q|^{-\alpha/n} \| \chi_Q \|_{L^{p(\cdot)}_{\omega}}$.

\begin{lemma} (\cite[Lemma 4]{Rocha4}) \label{2Q}
Let $p(\cdot) : \mathbb{R}^{n} \to (0, \infty)$ be a measurable function with $0 < p_{-} \leq p_{+} < \infty$. If 
$\omega \in \mathcal{W}_{p(\cdot)}$, then, for every cube $Q \subset \mathbb{R}^{n}$,
\[
\| \chi_{2Q} \|_{L^{p(\cdot)}_\omega}  \approx \| \chi_{Q} \|_{L^{p(\cdot)}_\omega} .
\]
\end{lemma}

\begin{lemma} (\cite[Theorem 3]{Rocha4})\label{Feff-Stein ineq}
Let $0 \leq \alpha < n$, $1 < u < \infty$ and let $q(\cdot) : \mathbb{R}^{n} \to (0, \infty)$ be a measurable function with 
$0 < q_{-} \leq q_{+} < \infty$. If $\omega \in \mathcal{W}_{q(\cdot)}$, then for 
$\frac{1}{p(\cdot)} := \frac{1}{q(\cdot)} + \frac{\alpha}{n}$ and any $s > s_{\omega, \, q(\cdot)} + \frac{\alpha}{n}$,
\begin{equation} \label{maximal fract ineq}
\left\| \left( \sum_{j} (M_{\frac{\alpha}{s}}f_j)^{u} \right)^{1/u} \right\|_{L^{sq(\cdot)}_{\omega^{1/s}}} \lesssim 
\left\| \left( \sum_{j} |f_j|^{u} \right)^{1/u} \right\|_{L^{sp(\cdot)}_{\omega^{1/s}}},
\end{equation}
holds for all sequences of bounded measurable functions with compact support $\{ f_j \}_{j=1}^{\infty}$.
\end{lemma}

\begin{lemma} \label{q0 estimate}
Let $q(\cdot) : \mathbb{R}^{n} \to (0, \infty)$ be a measurable function with $0 < q_{-} \leq q_{+} < \infty$. If 
$q_0 > \max \{1, q_{+} \}$ and $\omega \in \mathcal{W}_{q(\cdot)}$, then for any countable collection of cubes $\{ Q_j \}$ and non-negative functions $h_j$ such that
$\supp(h_j) \subset Q_j$
\begin{equation} \label{vector ineq}
\left\Vert  \sum_{j=1}^{\infty}  h_j \right\Vert_{L^{q(\cdot)}_{\omega}(\mathbb{R}^n)} \lesssim 
\left\Vert \sum_{j=1}^{\infty} \left( \frac{1}{|Q_j|}\int_{Q_j} h_j^{q_0} \right)^{1/q_0} \chi_{Q_j}  \right\Vert_{L^{q(\cdot)}_{\omega}(\mathbb{R}^n)}.
\end{equation}
\end{lemma}

\begin{proof}
We apply \cite[Lemma 4.9]{Uribe} with $g_j = \omega \cdot h_j$. Then, (\ref{vector ineq}) follows from (\ref{quasi}).
\end{proof}

\begin{lemma} \label{Qalfan}
Let $0 < \alpha < n$ and let $q(\cdot) : \mathbb{R}^{n} \to (0, \infty)$ be a measurable function with $0 < q_{-} \leq q_{+} < \infty$. If 
$\omega \in \mathcal{W}_{q(\cdot)}$ and $\frac{1}{p(\cdot)} := \frac{1}{q(\cdot)} + \frac{\alpha}{n}$, then for any countable collection of cubes $\{ Q_j \}$ and $\lambda_j > 0$
\begin{equation} \label{vector ineq 2}
\left\Vert  \sum_{j=1}^{\infty} \lambda_j |Q_j|^{\frac{\alpha}{n}} \chi_{Q_j} \right\Vert_{L^{q(\cdot)}_{\omega}(\mathbb{R}^n)} \lesssim
\left\Vert \sum_{j=1}^{\infty} \lambda_j \chi_{Q_j}  \right\Vert_{L^{p(\cdot)}_{\omega}(\mathbb{R}^n)}.
\end{equation}
\end{lemma}

\begin{proof}
Given $\omega \in \mathcal{W}_{q(\cdot)}$, by Definition \ref{pesos Wp}, there exist $s > 1/q_{-}$ and $\kappa >1$ such that the Hardy-Littlewood maximal operator $M$ is bounded on $L^{(sq(\cdot))'/\kappa}_{\omega^{-\kappa/s}}(\mathbb{R}^{n})$. Then, by Jensen's inequality we have 
$(Mf)^{\kappa} \leq M(|f|^{\kappa})$, so $M$ results bounded on $L^{(sq(\cdot))'}_{\omega^{-1/s}}(\mathbb{R}^{n})$. Now, for $0 < \alpha < n$, define
\[
\mathcal{F}_{\alpha} = \left\{ \left( \sum_{j=1}^{N} \lambda_j |Q_j|^{\frac{\alpha}{n}} \chi_{Q_j}, 
\sum_{j=1}^{N} \lambda_j \chi_{Q_j} \right) : N \in \mathbb{N}, \lambda_j >0, Q_j \in \mathcal{Q}  \right\},
\]
where $\mathcal{Q}$ denotes the set of all the cubes of $\mathbb{R}^n$.

Let $q_0 = \frac{1}{s}$ and let $p_0$ be defined by $\frac{1}{p_0} := \frac{1}{q_0} + \frac{\alpha}{n}$. Now, for any $v \in \mathcal{A}_1$ one has that $v^{p_0/q_0} \in RH_{q_0/p_0}$ (for the definition of the $\mathcal{A}_{1}$ class and the set $RH_{q_0/p_0}$, the reader may refer to \cite[Chapter 7]{Grafa}). Then, by \cite[Lemma 4.10]{Uribe} (applied with $w^{p_0} = v^{p_0/q_0}$), there exists an universal constant $C > 0$ such that for any $(F, G) \in \mathcal{F}_{\alpha}$ and any $v \in \mathcal{A}_1$
\begin{equation} \label{weighted fract ineq}
\int [F(x)]^{q_0} v(x) \, dx \leq C \left( \int [G(x)]^{p_0} [v(x)]^{p_0/q_0} \, dx \right)^{q_0/p_0}.
\end{equation}
On the other hand, by Lemma \ref{potencia s} - (iv) and Proposition \ref{norma equivalente}, we have
\begin{equation} \label{norma}
\| F \|^{q_0}_{L^{q(\cdot)}_{\omega}} = \| F^{q_0} \|_{L^{sq(\cdot)}_{\omega^{1/s}}}
\leq C \sup \left\{ \int_{\mathbb{R}^{n}} \left| [F(x)]^{q_0} g(x) \right| dx : \| g \|_{L^{(sq(\cdot))'}_{\omega^{-1/s}}} \leq 1 \right\}
\end{equation}
for some constant $C > 0$.

Let $\mathcal{R}$ be the operator defined on $L^{(sq(\cdot))'}_{\omega^{-1/s}}$ by
\[
\mathcal{R}g(x) = \sum_{k=0}^{\infty} \frac{M^{k}g(x)}{2^{k} \| M \|_{L^{(sq(\cdot))'}_{\omega^{-1/s}}}^{k}},
\]
where, for $k \geq 1$, $M^{k}$ denotes $k$ iterations of the Hardy-Littlewood maximal operator $M$, $M^{0} = M$, and 
$\| M \|_{L^{(sq(\cdot))'}_{\omega^{-1/s}}}$ is the operator norm of the maximal operator $M$ on 
$L^{(sq(\cdot))'}_{\omega^{-1/s}}$. It follows immediately from this definition that:

$(i)$ if $g$ is non-negative, $g(x) \leq \mathcal{R}g(x)$ a.e. $x \in \mathbb{R}^{n}$;

$(ii)$ $\| \mathcal{R}g \|_{L^{(sq(\cdot))'}_{\omega^{-1/s}}} \leq 2 \| g \|_{L^{(sq(\cdot))'}_{\omega^{-1/s}}}$; 

$(iii)$ $\mathcal{R}g \in \mathcal{A}_1$ with $[\mathcal{R}g]_{\mathcal{A}_1} \leq 2 \| M \|_{L^{(sq(\cdot))'}_{\omega^{-1/s}}}$.
\\
Since $F$ is non-negative, we can take the supremum in (\ref{norma}) over those non-negative $g$ only. For any fixed non-negative 
$g \in L^{(sq(\cdot))'}_{\omega^{-1/s}}$, by $(i)$ above we have that
\begin{equation} \label{int g}
\int [F(x)]^{q_0} g(x) dx \leq \int [F(x)]^{q_0} (\mathcal{R}g)(x) dx.
\end{equation}
Being $\frac{1}{p(\cdot)} - \frac{1}{q(\cdot)} = \frac{1}{p_0} - \frac{1}{q_0}$ and $q_0 = \frac{1}{s}$, we have 
$\frac{p_0}{q_0} \left( \frac{1}{p_0} p(\cdot) \right)' = \left( \frac{1}{q_0} q(\cdot) \right)' = (sq(\cdot))'$. Then, $(iii)$, 
(\ref{weighted fract ineq}), H\"older's inequality and Lemma \ref{potencia s} - (iv) yield
\begin{equation} \label{int Rg}
\int [F(x)]^{q_0} (\mathcal{R}g)(x) dx \leq C \left( \int [G(x)]^{p_0} [(\mathcal{R}g)(x)]^{p_0 / q_0} dx \right)^{q_0/p_0} 
\end{equation}
\[
\leq C \| G^{p_0} \|_{L^{p(\cdot)/p_0}_{\omega^{p_0}}}^{q_0/p_0} 
\|(\mathcal{R}g)^{p_0/q_0} \|_{L^{(p(\cdot)/p_0)'}_{\omega^{-p_0}}}^{q_0/p_0}
\]
\[
= C \| G \|^{q_0}_{L^{p(\cdot)}_{\omega}} 
\|\mathcal{R}g \|_{L^{\frac{p_0}{q_0} \left(\frac{p(\cdot)}{p_0} \right)'}_{\omega^{-q_0}}}
\]
\[
= C \| G \|^{q_0}_{L^{p(\cdot)}_{\omega}} \| \mathcal{R}g \|_{L^{(sq(\cdot))'}_{\omega^{-1/s}}}
\]
now, $(ii)$ gives
\[
\leq C \| G \|^{q_0}_{L^{p(\cdot)}_{\omega}} \| g \|_{L^{(sq(\cdot))'}_{\omega^{-1/s}}}.
\]
Thus, for every $(F,G) \in \mathcal{F}_{\alpha}$ fixed, (\ref{int g}) and (\ref{int Rg}) lead to
\begin{equation} \label{norma2}
\int [F(x)]^{q_0} g(x) dx \leq C \| G \|^{q_0}_{L^{p(\cdot)}_{\omega}},
\end{equation}
for all non-negative $g$ with $\| g \|_{L^{(sq(\cdot))'}_{\omega^{-1/s}}} \leq 1$. Then, (\ref{norma}) and (\ref{norma2}) give 
\begin{equation} \label{vector ineq 3}
\left\Vert  \sum_{j=1}^{N} \lambda_j |Q_j|^{\frac{\alpha}{n}} \chi_{Q_j} \right\Vert_{L^{q(\cdot)}_{\omega}(\mathbb{R}^n)} \lesssim
\left\Vert \sum_{j=1}^{N} \lambda_j \chi_{Q_j}  \right\Vert_{L^{p(\cdot)}_{\omega}(\mathbb{R}^n)},
\end{equation}
for every $N \geq 1$. Finally, by passing to the limit in (\ref{vector ineq 3}), we obtain (\ref{vector ineq 2}).
\end{proof}

\begin{remark} \label{hyp A2}
From Lemma \ref{Qalfan} and Lemma \ref{potencia s} - (ii), it is clear that for $\omega \in \mathcal{W}_{q(\cdot)}$, $\frac{1}{p(\cdot)} := \frac{1}{q(\cdot)} + 
\frac{\alpha}{n}$ and every cube $Q \subset \mathbb{R}^{n}$,  
\[
\| \chi_Q \|_{L^{q(\cdot)}_{\omega}} \lesssim |Q|^{-\alpha/n} \| \chi_Q \|_{L^{p(\cdot)}_{\omega}}.
\]
An open question if the opposite inequality holds. If this is proved, then the condition $A2)$ becomes in a theorem.
\end{remark}

\section{Main result}

In this section we will re-establish the $H^{p(\cdot)}_{\omega}(\mathbb{R}^{n}) \to L^{q(\cdot)}_{\omega}(\mathbb{R}^{n})$ and $H^{p(\cdot)}_{\omega}(\mathbb{R}^{n}) \to H^{q(\cdot)}_{\omega}(\mathbb{R}^{n})$ boundedness for the Riesz potential without taking on the hypothesis $A2)$.

\begin{theorem} \label{Hpw-Lqw and Hpw-Hqw}
Let $0 < \alpha < n$, $q(\cdot) \in \mathcal{P}^{\log}(\mathbb{R}^{n})$ with $0 < q_{-} \leq q_{+} < \infty$, and 
$\omega \in \mathcal{W}_{q(\cdot)}$. If $\frac{1}{p(\cdot)} := \frac{1}{q(\cdot)} + \frac{\alpha}{n}$, then the Riesz potential $I_{\alpha}$ given by (\ref{Ia}) can be extended to a bounded operator $H^{p(\cdot)}_{\omega}(\mathbb{R}^{n}) \to L^{q(\cdot)}_{\omega}(\mathbb{R}^{n})$ and $H^{p(\cdot)}_{\omega}(\mathbb{R}^{n}) \to H^{q(\cdot)}_{\omega}(\mathbb{R}^{n})$.
\end{theorem}

\begin{proof}
We first prove that the operator $I_{\alpha}$ can be extended to a bounded operator 
$H^{p(\cdot)}_{\omega}(\mathbb{R}^{n}) \to L^{q(\cdot)}_{\omega}(\mathbb{R}^{n})$. Given $\omega \in \mathcal{W}_{q(\cdot)}$, by 
Definition \ref{pesos Wp}, there exists $0 < \theta < 1$ such that $\frac{1}{\theta} \in \mathbb{S}_{\omega, \, q(\cdot)}$. Now, we 
take $q_0 > \max \{ \frac{n}{n - \alpha}, q_{+} \}$ and define $\frac{1}{p_0} := \frac{1}{q_0} + \frac{\alpha}{n}$. By Proposition 
\ref{WqcWp}, for $\frac{1}{p(\cdot)} = \frac{1}{q(\cdot)} + \frac{\alpha}{n}$, we have that $\mathcal{W}_{q(\cdot)} \subset 
\mathcal{W}_{p(\cdot)}$ and $s_{\omega, \, p(\cdot)} \leq s_{\omega, \, q(\cdot)} + \frac{\alpha}{n}$. So, given 
$f \in S_{0}(\mathbb{R}^{n})$, from Theorem \ref{w atomic decomp} and since one can always choose atoms with additional vanishing moment, we have that there exist a sequence of real numbers $\{\lambda_j\}_{j=1}^{\infty}$, a sequence of cubes $Q_j = Q(z_j,r_j)$ centered at $z_j$ with side length $r_j$ and $\omega - (p(\cdot), p_0, \lfloor n s_{\omega, \, q(\cdot)} + \alpha - n \rfloor)$ atoms $a_j$ supported on $Q_j$, satisfying
\begin{equation} \label{ineq Hpw}
\left\| \sum_{j} \left( \frac{|\lambda_j |}{\| \chi_{Q_j} \|_{L^{p(\cdot)}_{\omega}}} \right)^{\theta} \chi_{Q_j} \right\|_{L^{p(\cdot)/\theta}_{\omega^{\theta}}}^{1/\theta} \lesssim \| f \|_{H^{p(\cdot)}_{\omega}},
\end{equation}
and $f = \sum_{j} \lambda_j a_j$ converges in $L^{p_0}(\mathbb{R}^{n})$. By Sobolev's Theorem we have that 
$I_{\alpha }$ is bounded from  $L^{p_{0}}\left( \mathbb{R}^{n}\right) $ into $L^{q_{0}}\left( \mathbb{R}^{n}\right)$, so
\[
|I_{\alpha}f(x)| \leq \sum_{j} |\lambda_j| |I_{\alpha}a_j(x)|, \,\,\,\,\, \textit{a.e.} \, x \in \mathbb{R}^{n}.
\]
Then,
\begin{equation} \label{U1_U2}
\| I_{\alpha}f \|_{L^{q(\cdot)}_{\omega}} \lesssim \left\| \sum_{j} |\lambda_j| \, \chi_{2 Q_j} \cdot I_{\alpha} a_j 
\right\|_{L^{q(\cdot)}_{\omega}} + \left\| \sum_{j} |\lambda_j| \, \chi_{\mathbb{R}^{n} \setminus 2 Q_j} \cdot I_{\alpha} a_j 
\right\|_{L^{q(\cdot)}_{\omega}} = U_1 + U_2,
\end{equation}
where $2Q_j = Q(z_j, 2r_j)$. To estimate $U_1$, we first apply Sobolev's theorem to the expression 
$\chi_{2 Q_j} \cdot I_{\alpha} a_j $ followed by Lemma \ref{2Q} and obtain
\[
\left\| I_{\alpha} a_j \right\|_{L^{q_{0}}(2Q_{j})} \lesssim \left\| a_j \right\|_{L^{p_{0}}} \lesssim 
\frac{| Q_j |^{\frac{1}{p_{0}}}}{\left\| \chi _{Q_j }\right\|_{L^{p(\cdot)}_{\omega}}} \lesssim 
\frac{\left| 2Q_{j} \right|^{\frac{1}{q_{0}}+\frac{\alpha}{n}}}{\left\| \chi_{2Q_{j}} \right\|_{L^{p(\cdot)}_{\omega}}},
\]
so
\begin{equation} \label{average}
\left( \frac{1}{|2Q_{j}|} \int_{2Q_j} |I_{\alpha} a_j|^{q_0}  \right)^{1/q_0} \lesssim 
\frac{\left| 2Q_{j} \right|^{\frac{\alpha}{n}}}{\left\| \chi_{2Q_{j}} \right\|_{L^{p(\cdot)}_{\omega}}}.
\end{equation}
Now, Lemma \ref{q0 estimate}, (\ref{average}), Lemma \ref{Qalfan}, Lemma \ref{2Q}, $0 < \theta < 1$ and (\ref{ineq Hpw})  lead to
\begin{equation} \label{U1}
U_1 = \left\| \sum_{j} |\lambda_j| \, \chi_{2 Q_j} \cdot I_{\alpha} a_j  \right\|_{L^{q(\cdot)}_{\omega}} \lesssim 
\left\| \sum_{j} |\lambda_j| \left( \frac{1}{|2Q_{j}|} \int_{2Q_j} |I_{\alpha} a_j|^{q_0}  \right)^{1/q_0} \chi_{2 Q_j} 
\right\|_{L^{q(\cdot)}_{\omega}} 
\end{equation}
\[
\lesssim \left\| \sum_{j} |\lambda_j| \frac{\left| 2Q_{j} \right|^{\frac{\alpha}{n}}}{\left\| \chi_{2Q_{j}} \right\|_{L^{p(\cdot)}_{\omega}}} \chi_{2 Q_j} \right\|_{L^{q(\cdot)}_{\omega}} \lesssim
\left\| \sum_{j}  \frac{|\lambda_j|}{\left\| \chi_{2Q_{j}} \right\|_{L^{p(\cdot)}_{\omega}}} \chi_{2 Q_j} \right\|_{L^{p(\cdot)}_{\omega}}
\]
\[
\lesssim \left\| \sum_{j} \left( \frac{|\lambda_j |}{\| \chi_{Q_j} \|_{L^{p(\cdot)}_{\omega}}} \right)^{\theta} \chi_{Q_j} 
\right\|_{L^{p(\cdot)/\theta}_{\omega^{\theta}}}^{1/\theta} \lesssim \| f \|_{H^{p(\cdot)}_{\omega}}.
\]
To estimate $U_2$, let $N := \lfloor n s_{\omega, \, q(\cdot)} + \alpha - n \rfloor$, and let $a_j(\cdot)$ be a 
$\omega - (p(\cdot), p_0, N)$ atom supported on the cube $Q_j = Q(z_j, r_j)$. In view of the moment condition $a3)$ of $a_j(\cdot)$, we obtain
\[
I_{\alpha}a_j(x) = \int_{Q_j} \left( |x-y|^{\alpha- n} - q_{N}(x,y) \right) a_j(y) dy, \,\,\,\,\,\, \textit{for all} \,\,\, x \notin 2Q_j,
\]
where $q_{N}$ is the degree $N$ Taylor polynomial of the function $y \rightarrow |x-y|^{\alpha - n}$ expanded around $z_j$. By the standard estimate of the remainder term of the Taylor expansion, for any $y  \in Q_j$ and any $x \notin 2Q_j$, we get
\[
\left||x-y|^{\alpha- n} - q_{N}(x,y) \right| \leq C r^{N +1}_j |x - z_j|^{-n+ \alpha -N -1},
\]
this inequality and the condition $a_2)$ of the atom $a(\cdot)$ allow us to conclude that
\begin{equation} \label{Ialfa}
|I_{\alpha}a_j(x) | \lesssim \frac{r^{n+N+1}_j}{\| \chi_{Q_j} \|_{L^{p(\cdot)}_\omega}} |x-z_j|^{-n + \alpha - N -1}
\lesssim \frac{\left[ M_{\frac{\alpha n}{n+N+1}}(\chi_{Q_j}) (x) \right]^{\frac{n+N+1}{n}}}{\| \chi_{Q_j} \|_{L^{p(\cdot)}_\omega}},
\end{equation}
for all $x \notin 2Q_j$. Putting $s= \frac{n+N+1}{n}$, (\ref{Ialfa}) leads to
\[
U_2 \lesssim \left\| \left\{ \sum_j  \frac{|\lambda_j|}{\| \chi_{Q_j} \|_{L^{p(\cdot)}_\omega}} \left[ M_{\frac{\alpha}{s}}(\chi_{Q_j}) \right]^{s} \right\}^{1/s} \right\|^{s}_{L^{sq(\cdot)}_{\omega^{1/s}}}. 
\]
Since
\[
s = \frac{n + \lfloor n s_{\omega, \, q(\cdot)} + \alpha - n \rfloor + 1}{n} > s_{\omega, \, q(\cdot)} + \frac{\alpha}{n},
\]
to apply Lemma \ref{Feff-Stein ineq}, with $u=s$, we obtain
\[
U_2 \lesssim \left\| \left\{ \sum_j  \frac{|\lambda_j|}{\| \chi_{Q_j} \|_{L^{p(\cdot)}_\omega}} \chi_{Q_j} \right\}^{1/s} 
\right\|^{s}_{L^{sp(\cdot)}_{\omega^{1/s}}} = \left\| \sum_j \frac{|\lambda_j|}{\| \chi_{Q_j} \|_{L^{p(\cdot)}_\omega}} \chi_{Q_j}
\right\|_{L^{p(\cdot)}_{\omega}}. 
\]
Being $0 < \theta < 1$, the $\theta$-inequality and (\ref{ineq Hpw}) give
\begin{equation} \label{U2}
U_2 \lesssim \left\| \sum_j \left(\frac{|\lambda_j|}{\| \chi_{Q_j} \|_{L^{p(\cdot)}_\omega}} \right)^{\theta} \chi_{Q_j}
\right\|^{1/\theta}_{L^{p(\cdot)/\theta}_{\omega^{\theta}}} \lesssim \| f \|_{H^{p(\cdot)}_{\omega}}.
\end{equation}
Hence, (\ref{U1_U2}), (\ref{U1}) and (\ref{U2}) yield
\[
\| I_{\alpha} f \|_{L^{q(\cdot)}_{\omega}} \lesssim \| f \|_{H^{p(\cdot)}_{\omega}}, \,\,\, \text{for all} \,\, 
f \in \mathcal{S}_{0}(\mathbb{R}^{n}).
\]
Then, by Propositions \ref{WqcWp} and \ref{dense}, it follows that $I_{\alpha}$ extends to a bounded operator 
$H^{p(\cdot)}_{\omega}(\mathbb{R}^{n}) \to L^{q(\cdot)}_{\omega}(\mathbb{R}^{n})$.

Now, by using Definition \ref{Def Hpw}, we will prove that the Riesz potential $I_{\alpha}$ extends to a bounded operator $H^{p(\cdot)}_{\omega}(\mathbb{R}^{n}) \to H^{q(\cdot)}_{\omega}(\mathbb{R}^{n})$. For them, we consider $N :=  \lfloor n s_{\omega, \, q(\cdot)} + \alpha - n \rfloor$. Since $I_{\alpha}$ is bounded from $L^{p_0}(\mathbb{R}^{n})$ into $L^{q_0}(\mathbb{R}^{n})$, $H^{q_0}(\mathbb{R}^{n}) \equiv L^{q_0}(\mathbb{R}^{n})$ with comparable norms, and $I_{\alpha}f = \sum_{j} \lambda_j I_{\alpha}a_j$ converges in $L^{q_0}(\mathbb{R}^{n})$, it follows that
\[
\mathcal{M}_{N}(I_{\alpha} f)(x) \leq \sum_{j=1}^{\infty} |\lambda_j| \mathcal{M}_{N}(I_{\alpha} a_j)(x), \,\,\, a.e. \,\, x \in \mathbb{R}^{n}.
\]
Then,
\[
\| I_{\alpha} f \|_{H^{q(\cdot)}_{\omega}} = \| \mathcal{M}_N(I_{\alpha} f) \|_{L^{q(\cdot)}_{\omega}} \leq \left\| \sum_{j} |\lambda_j| \chi_{2 Q_{j}} \mathcal{M}_N(I_{\alpha} a_j) \right\|_{L^{q(\cdot)}_{\omega}}  
\]
\[
+ 
\left\| \sum_j |\lambda_j| \chi_{\mathbb{R}^{n} \setminus 2 Q_{j}} \mathcal{M}_N(T_{\alpha} a_j) \right\|_{L^{q(\cdot)}_{\omega}} =: V_1 + V_2.
\]
To estimate $V_1$, we observe, by Sobolev's Theorem and Lemma \ref{2Q}, that 
\[
\left\| \mathcal{M}_N (I_{\alpha} a_j) \right\|_{L^{q_{0}}(2Q_{j})} \lesssim 
\left\| I_{\alpha} a_j \right\|_{L^{q_{0}}} \lesssim \left\| a_j \right\|_{L^{p_{0}}}
\lesssim 
\frac{| Q_j |^{\frac{1}{p_{0}}}}{\left\| \chi _{Q_j }\right\|_{L^{p(\cdot)}_{\omega}}} \lesssim 
\frac{\left| 2Q_{j} \right|^{\frac{1}{q_{0}}+\frac{\alpha}{n}}}{\left\| \chi_{2Q_{j}} \right\|_{L^{p(\cdot)}_{\omega}}}.
\]
Then, by proceeding as above in the estimate of $U_1$, we get
\[
V_1 \lesssim \left\| \sum_{j=1}^{\infty} \left( \frac{|\lambda_j |}{\| \chi_{Q_j} \|_{L^{p(\cdot)}_{\omega}}} \right)^{\theta} \chi_{Q_j} \right\|_{L^{p(\cdot)/\theta}_{\omega^{\theta}}}^{1/\theta} \lesssim \| f \|_{H^{p(\cdot)}_{\omega}}.
\]
Now, we estimate $V_2$. We put $K_{\alpha}(y)=|y|^{\alpha - n}$ and consider $\phi \in \mathcal{S}(\mathbb{H}^{n})$ with $\| \phi \|_{\mathcal{S}(\mathbb{H}^{n}), \, N} \leq 1$. Then, for $x \notin 2 Q_j$ and every $t > 0$, by the moment condition $a3)$ of the atoms, we have
\[
((I_{\alpha} a_j) \ast \phi_t)(x) = \int_{Q_j} a_j(y) (K_{\alpha} \ast \phi_t)(x-y) \, dy  
\]
\[
= \int_{Q_j} a_j(y) \left[(K_{\alpha} \ast \phi_t)(x-y) - q_{x, t}(y) \right] \, dy, 
\] 
where $y \to q_{x, t}(y)$ is the Taylor polynomial of the function $y \to (K_{\alpha} \ast \phi_t)(x-y)$ at 
$z_j$ of degree $N$. To apply \cite[Lemma 6.9]{Foll} with $G=\mathbb{R}^n$, $r = N +1$ and $K = K_{\alpha}$, we have
\[
|\partial^{\beta}(K_{\alpha} \ast \phi_t)(u)| = |((\partial^{\beta} K_{\alpha}) \ast \phi_t)(u)|
\lesssim |u|^{\alpha - n - |\beta|},
\]
for all $u \neq 0$, $t > 0$ and $|\beta| \leq N+1$. Then, by the standard estimate of the remainder term of the Taylor expansion, for any 
$y  \in Q_j$ and any $x \notin 2Q_j$, we obtain
\[
\left| (K_{\alpha} \ast \phi_t)(x-y) - q_{x, t}(y) \right| \lesssim r^{N+1}_j  |x-z_j|^{-n+ \alpha - N -1}.
\]
This estimate does not depend on $t$. Finally, according to the ideas to estimate $U_2$ above and taking the supremum on $t>0$ and $\phi \in \mathcal{F}_N$, we obtain
\[
V_2 \lesssim \left\| \sum_{j=1}^{\infty} \left( \frac{|\lambda_j |}{\| \chi_{Q_j} \|_{L^{p(\cdot)}_{\omega}}} \right)^{\theta} \chi_{Q_j} \right\|_{L^{p(\cdot)/\theta}_{\omega^{\theta}}}^{1/\theta} \lesssim \| f \|_{H^{p(\cdot)}_{\omega}},
\]
for all $f \in \mathcal{S}_{0}(\mathbb{R}^{n})$. Thus, by Propositions \ref{WqcWp} and \ref{dense}, $I_{\alpha}$ extends to a bounded operator 
$H^{p(\cdot)}_{\omega}(\mathbb{R}^{n}) \to H^{q(\cdot)}_{\omega}(\mathbb{R}^{n})$.
\end{proof}

%{\bf Acknowledgements.} The author is very grateful to the referee for the useful suggestions and comments
%that improved the original manuscript.

\end{document}